\documentclass[12pt]{article}
\font\smallit=cmti10

\makeatletter

\renewcommand\section{\@startsection {section}{1}{\z@}
	{-30pt \@plus -1ex \@minus -.2ex}
	{2.3ex \@plus.2ex}
	{\normalfont\normalsize\bfseries}}

\renewcommand\subsection{\@startsection{subsection}{2}{\z@}
	{-3.25ex\@plus -1ex \@minus -.2ex}
	{1.5ex \@plus .2ex}
	{\normalfont\normalsize\bfseries}}

\renewcommand{\@seccntformat}[1]{\csname the#1\endcsname. }

\makeatother

\newtheorem{thm}{Theorem}[section]

\newenvironment{proof}{\par\noindent\textsc{Proof}}{\medskip}
\def\qed{\relax\ifmmode\hskip2em \Box\vspace{-7pt}
	\else\unskip\nobreak\hskip1em $\Box$\fi}

\usepackage{amsmath,amsfonts,amssymb,latexsym}

\usepackage{enumerate} %example \begin{enumerate}[1)]

\renewcommand{\arraycolsep}{1pt}

\begin{document}
	\vskip 40pt
	
	\begin{center}
		\uppercase{\bf Steinhaus Triangles Generated by \\
			Vectors of the Canonical Bases}\\
		\vskip 20pt
		{\bf Josep M. Brunat and
			Montserrat Maureso} \\
		{\smallit Departament de Matem\`atica Aplicada, Universitat
			Polit\`ecnica  de Catalunya, Barcelona, Catalunya}\\
		{\tt josep.m.brunat@upc.edu,
			montserrat.maureso@upc.edu}\\ %(optional)
		\vskip 10pt
	\end{center}
	\vskip 30pt
	
	\centerline{\bf Abstract}
	We give a method to calculate the weight of binary Steinhaus triangles
	generated by the vectors of the canonical basis of the vector space $\mathbb{F}_2^n$ over the field $\mathbb{F}_2$ of order $2$.
	\noindent

\section{Introduction}
In 1958 H.~Steinhaus~\cite{Steinhaus}(Chapter VII) posed the following problem:

\begin{quote}
 The figure given below consists of 14 plus signs and 14 minus
signs. They are arranged in such a way that under each pair of
equal signs there appears a positive sign and under opposite signs
a minus sign.
$$
\begin{array}{cccccccccccccc}
+ &   & + &   & - &   & + &   & - &   & + &   & + \\
  & + &   & - &   & - &   & - &   & - &   & + \\
  &   & - &   & + &   & + &   & + &   & - \\
  &   &   & - &   & + &   & + &   & - \\
  &   &   &   & - &   & + &   & - \\
  &   &   &   &   & - &   & - &  \\
  &   &   &   &   &   & +  \\
\end{array}
$$
If the first row had $n$ signs, then in an analogous figure there
would be $n(n+1)/2$ signs; our example corresponds to the case
$n = 7$. As $n(n+1)/2$ is an even number for $n=3, 4, 7, 8, 11,
12$ etc., we can ask whether it is possible to construct a figure
analogous to the above one and beginning with $n$ signs in the
highest row.
\end{quote}

The figure above  is called a \emph{Steinhaus triangle}. We can change
signs $+$ and $-$ by $1$ and $0$, respectively, and
reformulate the rule of signs by the sum in the field $\mathbb{F}_2$ of order $2$.  The number
of ones in a sequence is called its \emph{weight}, and the number of ones in
the whole triangle is called the \emph{weight} of the triangle. More formally, given
a sequence $\mathbf{x}=(x_0,\ldots,x_{n-1})\in \mathbb{F}_2^n$, its
\emph{derivative} is the sequence
$\partial\mathbf{x}=(x_0+x_1,x_1+x_2,\ldots,x_{n-2}+x_{n-1})\in\mathbb{F}_2^{n-1}$.
Recursively, for $r\ge 2$, we define
$\partial^r\mathbf{x}=\partial\partial\mathbf{x}^{r-1}$; we also define $\partial^0\mathbf{x}=\mathbf{x}$.
The \emph{Steinhaus triangle}
generated by $\mathbf{x}$ is the sequence
$T(\mathbf{x})=(\mathbf{x},\partial\mathbf{x},\partial^2\mathbf{x},\ldots,\partial^{n-1}\mathbf{x})$. The length $n$ of the initial sequence is the \emph{size} of the triangle. The
$r$-th derivative $\partial^r\mathbf{x}$ of $\mathbf{x}$ is the $r$-th row the
triangle (thus, we number rows from $0$ to $n-1$). We number the
coordinates of $\partial^r\mathbf{x}$ from $0$ to $n-1-r$.  The \emph{weight}
of the sequence $\mathbf{x}=(x_0,\ldots,x_{n-1})\in\mathbb{F}_2^n$
is $|\mathbf{x}|=\#\{i: x_i=1\}$, and the \emph{weight} of $T(\mathbf{x})$ is
$$
|T(\mathbf{x})|=\sum_{i=0}^{n-1}|\partial^i\mathbf{x}|.
$$
The problem posed by Steinhaus is the following one: given $n\equiv 0,3\pmod{4}$, determine if there exist sequences
$\mathbf{x}\in \mathbb{F}_2^n$ such that $|T(\mathbf{x})|=n(n+1)/4$.
H.~Harborth~\cite{Harborth} solved the problem
by constructing examples of such
sequences. The problem has been also solved for sequences $\mathbf{x}$ with
some additional conditions (S.~Elihaou and D.~Hachez~\cite{ElHa, ElHa2},
S.~Elihaou, J.~M.~Mart\'{\i}n and M.~P.~Revuelta~\cite{ElMaRe}) and also generalized
for sequences in the cyclic group $\mathbb{Z}_m$
(J.~C. Molluzzo~\cite{Molluzo}, J.~Chappelon~\cite{Chappelon,Chappelon4,Chappelon5},
J.~Chappelon and S.~Elihaou~\cite{ChEl}). Steinhaus triangles appear in the
context of cellular automata, see A.~Barb\'e~\cite{Barbe2,Barbe3,Barbe} and J. Chappelon~\cite{Chappelon4,Chappelon5}.
In this context, A. Barb\'e~\cite{Barbe} has
studied some properties related to symmetries.
Steinhaus triangles with rotational and dihedral symmetry are characterized by J.~M. Brunat and M. Maureso in~\cite{BrMa} and, in~\cite{BrMa2}, they give some results about the sequences that generate triangles with extreme weights. Also, there is an asymptotic result by
F.~M.~Malysev and E.~V.~Kutyreva~\cite{MaKu}, who
estimated the number of Steinhaus triangles (which they call Boolean
Pascal triangles, and take the initial sequence in the bottom of the triangle)
of sufficiently large size $n$ containing a given
number $\omega\le kn$ ($k>0$) of ones.  Not much else is known about the weight distribution of Steinhaus triangles.

An easy induction that involves only the binomial number recurrence show that the entry of a Steinhaus triangle $T(\mathbf{x})$ on row $r$ and column $c$ is
\begin{equation}
\label{general entry}
T(\mathbf{x})(r,c)=\sum_{i=0}^r {r\choose i}x_{c+i},
\end{equation}
so, it depends linearly on the entries of $\mathbf{x}$.
Thus, the set $S(n)$ of Steinhaus triangles of size $n$ is an
$\mathbb{F}_2$-vector space of dimension $n$ that can be identified with a vector subspace of $\mathbb{F}_2^{n(n+1)/2}$,
that is, with a (binary) linear code of length $n(n+1)/2$ and dimension $n$. Then, the weight distribution of Steinhaus triangles, or equivalently, how many triangles
exist of every possible weight, is a particular case of the general problem of weight distribution in linear codes.
This is, in general, a difficult problem, and it seems that it is also difficult for the particular case of
Steinhaus triangles.

Obviously, in $S(n)$ there exists only one Steinhaus triangle of weight
$0$, which is the one generated by the sequence
$\mathbf{0}=(0,\ldots,0)$. It is easy to see that the next possible
weight is $n$ and that it is generated by the sequences
$\mathbf{e}_0=(1,0,\ldots,0)$ and
$\mathbf{e}_{n-1}=(0,\ldots,0,1)$. H.~Harborth, in his
paper~\cite{Harborth} solving the original problem, observed also the
maximum weight. Let $\overline{110}[n]$ the sequence formed by the
first $n$ terms of the succession $(1,1,0,1,1,0,1,1,0,\ldots)$ (for
instance $\overline{110}[6]=(1,1,0,1,1,0)$,
$\overline{110}[7]=(1,1,0,1,1,0,1)$,
$\overline{110}[8]=(1,1,0,1,1,0,1,1)$). Define
$\mathbf{z}_1=\overline{110}[n]$, and, analogously,
$\mathbf{z}_2=\overline{101}[n]$,
$\mathbf{z}_3=\overline{011}[n]$. Then, the maximum weight among
Steinhaus triangles of size $n$ is
\begin{itemize}
\item $n(n+1)/3$ if $n\equiv 0,2\pmod{3}$, and the sequences that generate the triangles of this weight are
$\mathbf{z}_1$, $\mathbf{z}_2$ and $\mathbf{z}_3$.
\item $(n^2+n+1)/3$ if $n\equiv 1\pmod{3}$, and the sequences that generate the triangles of this weight are
$\mathbf{z}_1$ and  $\mathbf{z}_3$.
\end{itemize}

Here, we study the weight of Steinhaus triangles generated by the vectors of the canonical basis. The main tool is Lucas's Theorem, so let us recall it.
Given a prime number $p$, Lucas's Theorem gives a
way to calculate binomial numbers in $\mathbb{F}_p$, the field of order $p$.

 \begin{thm}[Lucas's Theorem] Let $p$ be a prime number,
 \begin{align*}
 &r=\alpha_tp^t+\alpha_{t-1}p^{t-1}+\cdots+\alpha_1 p+\alpha_0,\mbox{ and } \\
 &s=\beta_tp^t+\beta_{t-1}p^{t-1}+\cdots+\beta_1p+\beta_0,
\end{align*}
with $\alpha_i,\beta_i\in\{0,1,\ldots,p-1\}$ for $i\in\{0,\ldots,t\}$.
Then,
$$
{r\choose s}={\alpha_t\choose\beta_t}{\alpha_{t-1}\choose\beta_{t-1}}\cdots {\alpha_0\choose\beta_0}
\quad \mbox{in }\mathbb{F}_p.
$$
\end{thm}
See~\cite{Fine} for a  nice proof from N.~J.~Fine.

\section{Canonical Basis}

For $n\ge 1$ and $0\le k\le n-1$, the $k$-th vector of the canonical basis of $\mathbb{F}_2^n$ is
$$
\mathbf{e}_k^{(n)}=(0,\stackrel{k)}{\ldots},0,1,0,\stackrel{n-1-k)}{\ldots},0).
$$
In order to  simplify notation, we define $T(k,n)=T(\mathbf{e}_k^{(n)})$ and $w(k,n)=|T(k,n)|$.
We give a way to calculate $w(k,n)$ that depends on $t=\lfloor \log_2k\rfloor +1$ and on the quotient and the remainder of dividing $n$ by $2^t$.

We apply Proposition~\ref{general entry} to the vectors of the canonical basis of $\mathbb{F}_2^n$.
The coordinates of $\mathbf{e}_k^{(n)}$
are $x_k=1$ and  $x_{c+i}=0$ for $i\ne k-c$. Therefore,
$$
T(k,n)(r,c)={r\choose k-c}={r\choose r-k+c}.
$$
In particular, $T(k,n)(r,c)=0$ for $c>k$, that is, in any row of
$T(k,n)$, all entries in positions $c$ with $c>k$ are $0$.

As $T(0,n)(r,0)=1$ and $T(0,n)(r,c)=0$ for $c>0$, we have $w(0,n)=n$
for all $n\ge 1$. Thus, in the following, we assume $k\ge 1$.

Note that the triangles $T(k,n)$ and $T(n-1-k,n)$ are
symmetric with respect to the vertical line passing through the
bottom vertex of the triangle:
\begin{align*}
T(n-1-k,n)(r,n-1-r-c)=&{r \choose r-(n-1-k)+(n-1-r-c)}\\[3pt]
=&{r\choose k-c}=T(k,n)(r,c).
\end{align*}
As we are interested in the values of
$w(k,n)=|T(k,n)|$, and the triangles
$T(k,n)$ and $T(n-1-k,n)$ have the same
weight because of symmetry, we can assume $n\ge 2k+1\ge 3$.

For $m\in\{1,\ldots, n\}$, denote by $s(k,n,m)$ the sum of the weights of the rows $0,1,\ldots,m-1$ of $T(n,k)$.
Our main result is the following.

\begin{thm}
\label{main}
Let $k\ge 1$ and $n\ge 2k+1$ be integers. Let
$t=\lfloor \log_2 k\rfloor+1$, \
$q=\lfloor n/2^t \rfloor$,\
$r=n-2^tq$,\
$\lambda=s(k,k+1+2^t,2^t)$, and
$\mu=w(k,r+2^t)$.
Then,
$$
w(k,n)=(q-1)\lambda+\mu.
$$
\end{thm}
\begin{proof}
We have $2^{t-1}\le k<2^t$, hence $n\ge 2k+1\ge 2\cdot 2^t+1$. We claim that row $2^t$  of
$T(k,n)$ is just $\mathbf{e}_k^{(n-2^t)}$. Indeed, we have
$$
T(k,n)(2^t,c)={2^t\choose k-c}.
$$
If $c<k$, and $k-c=\alpha_a2^a+\cdots +\alpha_12+\alpha_0$, we have $a<2^t$ and some $\alpha_i=1$ because $k-c>0$.
By Lucas's Theorem,
$$
{2^t\choose k-c}={1\choose 0}{0\choose 0}\cdots {0\choose 0} {0\choose\alpha_a}\cdots{0\choose\alpha_i}
               \cdots {0\choose\alpha_0}=0
$$
because ${0\choose \alpha_i}={0\choose 1}=0$.
Also, we have $T(k,n)(2^t,k)={2^t\choose 0}=1$, and, for $c>k$, $T(k,n)(2^t,c)={2^t\choose k-c}=0$,
so $\partial^{2^t}\mathbf{e}_k^{(n)}=\mathbf{e}_k^{(n-2^t)}$.

Therefore, if we delete the first $2^t$ rows of the triangle
$T(k,n)$, we obtain the triangle $T(k, n-2^t)$. Then, $w(k,n)-w(k,n-2^t)=s(k,n,2^t)$.
In any row, all entries in positions $c>k$  are $0$. Then,
$s(k,n,2^t)=s(k,k+1+2^t,2^t)$. Let $q$ and $r$ be the quotient and the remainder of dividing $n$ by $2^t$. Then,
for a fixed $k$, the sequence
$$
w(k,r+2^t), \ w(k,r+2\cdot 2^t), \ w(k,r+3\cdot 2^t),\ \ldots,\  w(k,r+q\cdot 2^t)=w(k,n)
$$
is an arithmetic progression with difference $\lambda=s(k,k+1+2^t,2^t)$. If $\mu=w(k,r+2^t)$, we have
$w(k,n)=\lambda(q-1)+\mu$. \qed
\end{proof}

For instance, let us compute $w(6,203)$. We have $k=6$; $t=3$; $2^t=8$, $n=203=8\cdot 25+3$, that is $q=25$ and $r=3$.
Let us calculate $\lambda=s(k,k+1+2^t,2^t)=s(6,15,8)$ and  $\mu=w(k,r+2^t)=w(6,11)$:
$$
\begin{array}{cc}
\begin{array}{cccccccccccccccccccccccccccccc}
0 && 0 && 0 && 0 && 0 && 0 && 1 && 0 && 0 && 0 && 0 && 0 && 0 && 0 && 0 \\
& 0 && 0 && 0 && 0 && 0 && 1 && 1 && 0 && 0 && 0 && 0 && 0 && 0 && 0 \\
&&  0 && 0 && 0 && 0 && 1 && 0 && 1 && 0 && 0 && 0 && 0 && 0 && 0  \\
&&&   0 && 0 && 0 && 1 && 1 && 1 && 1 && 0 && 0 && 0 && 0 && 0 \\
&&&&    0 && 0 && 1 && 0 && 0 && 0 && 1 && 0 && 0 && 0 && 0 \\
&&&&&     0 && 1 && 1 && 0 && 0 && 1 && 1 && 0 && 0 && 0 \\
&&&&&&      1 && 0 && 1 && 0 && 1 && 0 && 1 && 0 && 0 \\
&&&&&&&       1 && 1 && 1 && 1 && 1 && 1 && 1 && 0
\end{array}
\\
\\
\begin{array}{ccccccccccccccccccccccccc}
0 && 0 && 0 && 0 && 0 && 0 && 1 && 0 && 0 && 0 && 0 \\
& 0 && 0 && 0 && 0 && 0 && 1 && 1 && 0 && 0 && 0  \\
&&  0 && 0 && 0 && 0 && 1 && 0 && 1 && 0 && 0  \\
&&&   0 && 0 && 0 && 1 && 1 && 1 && 1 && 0 \\
&&&&    0 && 0 && 1 && 0 && 0 && 0 && 1 \\
&&&&&     0 && 1 && 1 && 0 && 0 && 1 \\
&&&&&&      1 && 0 && 1 && 0 && 1 \\
&&&&&&&       1 && 1 && 1 && 1 \\
&&&&&&&&        0  && 0 && 0 \\
&&&&&&&&&          0  && 0 \\
&&&&&&&&&&            0
\end{array}
\end{array}
$$
We see that  $\lambda=s(6,15,8)=26$ and  $\mu=w(6,11)=21$. Then
$w(6,203)=\lambda(q-1)+\mu=26\cdot 24+21=645$.

Note that $\lambda=s(k,k+1+2^t,2^t)$ depends only on $k$, and not on
$n$. Moreover, $\mu=w(k,2^t+r)$ depends on $k$ and $r$, but $r$, as
the remainder of a division by $2^t$, is bounded by $2^t\le 2k$.

The tables bellow give the values $\lambda$ and $\mu$ for $k\in\{1,2,3,4,5,6,7\}$. As in Theorem~\ref{main},
$t=\lfloor \log_2 k\rfloor +1$ and $r$ is the remainder of the division of $n$ by $2^t$.

% $$
% \renewcommand{\arraycolsep}{5pt}
% \begin{array}{|l|l|l|l|l|l|}
% \hline
% k & t & \lambda & r & \mu & w(k,n) \\
% \hline\hline
% 1 & 0 & 3       & 0 & 2   & 3q-1\\
%   &   &         & 1 & 3   & 3q\\
% \hline
% 2 & 2 & 8       & 0 & 5   & 8q-3\\
%   &   &         & 1 & 7   & 8q-1\\
%   &   &         & 2 & 8   & 8q\\
%   &   &         & 3 & 11  & 8q+3\\
% \hline
% 3 & 2 & 9       & 0 & 4   & 9q-5\\
%   &   &         & 1 & 6   & 9q-3\\
%   &   &         & 2 & 8   & 9q-1\\
%   &   &         & 3 & 9   & 9q\\
% \hline
% 4 & 3 & 22      & 0 & 13   & 22q-9\\
%   &   &         & 1 & 17   & 22q-5\\
%   &   &         & 2 & 19   & 22q-3\\
%   &   &         & 3 & 21   & 22q-1\\
%   &   &         & 4 & 22   & 22q\\
%   &   &         & 5 & 27   & 22q+5\\
%   &   &         & 6 & 30   & 22q+8\\
%   &   &         & 7 & 33   & 22q+11\\
% \hline
% 5 & 3 & 24      & 0 & 13   & 24q-11\\
%   &   &         & 1 & 15   & 24q-9\\
%   &   &         & 2 & 19   & 24q-5\\
%   &   &         & 3 & 21   & 24q-3\\
%   &   &         & 4 & 23   & 24q-1\\
%   &   &         & 5 & 24   & 24q\\
%   &   &         & 6 & 30   & 24q+6\\
%   &   &         & 7 & 33   & 24q+9\\
% \hline
% \end{array}
% $$

$$
\renewcommand{\arraycolsep}{5pt}
\begin{array}{|l|ll|llll|llll|}
\hline
k       & \mathbf{1}    &    & \mathbf{2}    &      &    &      & \mathbf{3}    &      &      &  \\
\hline
t       & 1    &    & 2    &      &    &      & 2    &      &      &  \\
\hline
\lambda & 3    &    & 8    &      &    &      & 9    &      &      &  \\
\hline
r       & 0    & 1  & 0    & 1    & 2  & 3    & 0    & 1    & 2    & 3 \\
\hline
\mu     & 2    & 3  & 5    & 7    & 8  & 11   & 4    & 6    & 8    & 9 \\
\hline
% w(k,n)  & 3q-1 & 3q & 8q-3 & 8q-1 & 8q & 8q+3 & 9q-5 & 9q-3 & 9q-1 & 9q \\
\end{array}
$$
$$
\renewcommand{\arraycolsep}{5pt}
\begin{array}{|l|llllllll|llllllll|}
\hline
k       & \mathbf{4}     &       &       &       &     &       &      &
        & \mathbf{5}     &       &       &       &     &       &      &     \\
\hline
t       & 3     &       &       &       &     &       &      &
        & 3     &       &       &       &     &       &      &    \\
\hline
\lambda & 22    &       &       &       &     &       &      &
        & 24    &       &       &       &     &       &      &  \\
\hline
r       & 0     & 1     & 2     & 3     & 4   & 5     & 6    & 7
        & 0     & 1     & 2     & 3     & 4   & 5     & 6    & 7     \\
\hline
\mu     & 13    & 17    & 19    & 21    & 22  & 27    & 30    & 33
        & 13    & 15    & 19    & 21    & 23  & 24    & 30    & 33  \\
\hline
\end{array}
$$
$$
\renewcommand{\arraycolsep}{5pt}
\begin{array}{|l|llllllll|llllllll|}
\hline
k       & \mathbf{6}     &       &       &       &     &       &      &
        & \mathbf{7}     &       &       &       &     &       &      &     \\
\hline
t       & 3     &       &       &       &     &       &      &
        & 3     &       &       &       &     &       &      &    \\
\hline
\lambda & 26    &       &       &       &     &       &      &
        & 27    &       &       &       &     &       &      &  \\
\hline
r       & 0     & 1     & 2     & 3     & 4   & 5     & 6    & 7
        & 0     & 1     & 2     & 3     & 4   & 5     & 6    & 7     \\
\hline
\mu     & 11    & 15    & 17    & 21    & 23  & 25    & 26    & 33
        &  8    & 12    & 16    & 18    & 22  & 24    & 26    & 27  \\
\hline
\end{array}
$$

We can see the relation
\begin{equation} \label{rec}
w(k,n)-w(k,n-2^t)=s(k,k+1+2^t,2^t)\end{equation}
 in another way. In fact, for a fixed $k$,
the weights $w(k,n)$ satisfy the  linear recurrence (\ref{rec}) of order $2^t$ with constant coefficients. The characteristic
polynomial is $x^{2^t}-1$, and its roots are the $2^t$-th roots of unity,
$$
\alpha^j=\cos(2\pi j/2^t)+i\sin(2\pi j/2^t),\quad j\in\{0,\ldots,2^t-1\}.
$$
As the right hand side of the recurrence is a constant, the solution of the recurrence is of the form
\begin{equation}
\label{solrec}
f(n)=A_0+A_1n+B_1\alpha^n+B_2\alpha^{2n}+\cdots+B_{2^t-1}\alpha^{(2^t-1)n},
\end{equation}
for certain complex constants $A_0,A_1,B_1,\ldots,B_{2^t-1}$. These constants can be determined
by finding the initial conditions $w(k, 2k+1)$, $w(k,2k+2)$, $\ldots$, $w(k,2k+2^t+1)$ and solving the system of the
$2^t+1$ equations $f(2k+j)=w(k, 2k+j)$, $j\in\{1,\ldots,2^t+1\}$ on the unknowns $A_0,A_1,B_1,\ldots,B_{2^t-1}$.
We skip the (long) calculations, but the results of this process for $1\le k\le 7$ are the following:
\begin{align*}
w(1,n) = & \frac{1}{4}\left(-5+6n+(-1)^n\right)\\[3mm]
w(2,n) = & \frac{1}{4}\left(-13+8n-(-1)^n+2\cos\frac{n\pi}{2} \right)\\[3mm]
w(3,n) = & \frac{1}{8}\left(-45+18n+3(-1)^n+2\cos\frac{n\pi}{2}+6\sin\frac{n\pi}{2}\right)
\\[3mm]
w(4,n) = & \frac{1}{8}\left(-71+22n-3(-1)^n+2(2+\sqrt{2})\cos\frac{n\pi}{4}\right.\\
         &\phantom{\frac{1}{8}\left(\right.} \left.+2\sin\frac{n\pi}{2}
           -6\cos\frac{n\pi}{2}+ 2(2-\sqrt{2})\cos\frac{3n\pi}{4} \right)
%\\[3mm]
\end{align*}
 \begin{align*}
 w(5,n)=&\frac{1}{16}\left(-196+48n+(4+3\sqrt{2})\cos\frac{n\pi}{4}+(2+3\sqrt{2})\sin\frac{n\pi}{4}\right.\\
        &\qquad -2\cos\frac{n\pi}{2}-6\sin\frac{n\pi}{2}\\
        &\qquad +(4-3\sqrt{2})\cos\frac{3n\pi}{4}+(-2+3\sqrt{2})\sin\frac{3n\pi}{4}\\
        &\qquad +8\cos n\pi\\
        &\qquad +(4-3\sqrt{2})\cos\frac{5n\pi}{4}+(2-3\sqrt{2})\sin\frac{5n\pi}{4}\\
        &\qquad -2\cos\frac{3n\pi}{2}+6\sin\frac{3n\pi}{2}\\
        &\qquad \left. +(4+3\sqrt{2})\cos\frac{7n\pi}{4}-(2+3\sqrt{2})\sin\frac{7n\pi}{4}\right)\\[3mm]
w(6,n)=&\frac{1}{8}\left(-128+26n+(1+\sqrt{2})\cos\frac{n\pi}{4}+(4+2\sqrt{2})\sin\frac{n\pi}{4}\right.\\
       &\qquad +4\cos\frac{n\pi}{2}-\sin\frac{n\pi}{2} \\
       &\qquad +(1-\sqrt{2})\cos\frac{3n\pi}{4}+(-4+2\sqrt{2})\sin\frac{3n\pi}{4}\\
       &\qquad -4\cos(n\pi)\\
       &\qquad +(1-\sqrt{2})\cos\frac{5n\pi}{4}+(4-2\sqrt{2})\sin\frac{5n\pi}{4}\\
       &\qquad +4\cos\frac{3n\pi}{2}+\sin\frac{3n\pi}{2}\\[3mm]
       &\qquad \left.+(1+\sqrt{2})\cos\frac{7n\pi}{4}+(-4-2\sqrt{2})\sin\frac{7n\pi}{4}\right)\\
w(7,n) =&\frac{1}{16}\left(-315+54n+(-1-3\sqrt{2})\cos\frac{n\pi}{4}+(7+6\sqrt{2})\sin\frac{n\pi}{4}\right.\\
        &\qquad +3\cos\frac{n\pi}{2}+9\sin\frac{n\pi}{2}\\
        &\qquad +(-1+3\sqrt{2})\cos\frac{3n\pi}{4}+(-7+6\sqrt{2})\sin\frac{3n\pi}{4}\\
        &\qquad +9\cos(n\pi)\\
        &\qquad +(-1+3\sqrt{2})\cos\frac{5n\pi}{4}+(7-6\sqrt{2})\sin\frac{5n\pi}{4}\\
        &\qquad +3\cos\frac{3n\pi}{2}-9\sin\frac{3n\pi}{2}\\
       &\qquad \left.+(-1-3\sqrt{2})\cos\frac{7n\pi}{4}+(-7-6\sqrt{2})\sin\frac{7n\pi}{4}\right).
\end{align*}

\end{document}